\documentclass{rae}
\usepackage{amsthm, amsmath, amssymb, amscd,  amsfonts, amstext,
  amsbsy}
\usepackage{mathrsfs}
\usepackage{enumerate}

\makeindex

\newcommand{\RR}{\ensuremath{\mathbb{R}}}

\def\leng{{\rm length}}
\def\diam{{\rm diam}}
\def\Cl{{\rm Cl}}

\newcommand{\Can}{\ensuremath{{}^\omega 2}}
\newcommand{\Bai}{\ensuremath{{}^\omega \omega}}
\newcommand{\ZFC}{\ensuremath{{\rm \mathsf{ZFC}}}}
\newcommand{\ZF}{\ensuremath{{\rm \mathsf{ZF}}}}
\newcommand{\DCR}{\ensuremath{{\rm \mathsf{DC}(\mathbb{R})}}}
\newcommand{\conc}{{}^\smallfrown}
\newcommand{\F}{\mathcal{F}}
\newcommand{\B}{\mathcal{B}}
\newcommand{\I}{\mathcal{I}}
\newcommand{\bSigma}{\mathbf{\Sigma}}
\newcommand{\bPi}{\mathbf{\Pi}}
\newcommand{\bDelta}{\mathbf{\Delta}}

\newcommand{\restr}[2]{#1 \restriction #2}
\newcommand{\seq}[2]{\langle #1 \mid  #2 \rangle}

\newtheorem{theorem}{Theorem}[section]
\newtheorem{lemma}[theorem]{Lemma}

\theoremstyle{definition}
\newtheorem*{claim}{Claim}
\newtheorem{defin}{Definition}

\title{A NEW PROOF OF A THEOREM OF JAYNE AND ROGERS}
\author{Luca Motto Ros\thanks{The author's research was supported by FWF (Austrian Research Fund) Grant P 19898-N18.},\
Kurt G\"odel Research Center for Mathematical Logic, University of Vienna, Austria
(e-mail:~luca.mottoros@libero.it)
\and
Brian Semmes,\
Institute for Logic, Language and Computation, University of Amsterdam, Holland
(e-mail:~B.T.Semmes@uva.nl)}
\date{\today}

\MathReviews{03E15, 54H05}
\begin{document}

\maketitle

\begin{abstract}
We give a new simpler proof of a theorem of Jayne and Rogers.
\end{abstract}

\section{Introduction}

In this paper we will give a new proof of a Jayne--Rogers theorem.
First recall from \cite{jaynerogers} the following definitions:

\begin{defin}
  Let $X,Y$ be metric spaces. A function $f \colon X \to Y$ is said to be 
\emph{$\bDelta^0_2$-function} if $f^{-1}(S) \in \bSigma^0_2$ for every 
$S \in \bSigma^0_2$ (equivalently, $f^{-1}(U) \in \bDelta^0_2$ for every 
open $U \subseteq Y$). Sometimes these functions are also called
\emph{first level Borel functions} (see \cite{jaynerogers}).

The function $f$ is said to be \emph{piecewise continuous} if $X$
 can be expressed as the union of an increasing sequence $X_0, X_1,
 \dotsc$ of closed sets such that $\restr{f}{X_n}$ is continuous for
 every $n \in \omega$.
\end{defin}

Obviously, if $f \colon X \to Y$ is piecewise continuous and $X' \subseteq X$ then $\restr{f}{X'}$ is piecewise continuous as well.
Observe also that $f$ is piecewise continuous if and only if there is
 a $\bDelta^0_2$-partition $\seq{D_n}{n \in \omega}$ of $X$
 such that $\restr{f}{D_n}$ is continuous for every $n \in
 \omega$. For one direction, if $f$ is piecewise continuous then putting
$D_0 = X_0$ and
 $D_{n+1} = X_{n+1} \setminus X_n$ we have the desired partition.
 Conversely, let $P_{m,n} \in \bPi^0_1$ be such that $D_n =
 \bigcup_{m \in \omega} P_{m,n}$ and $P_{m,n} \subseteq P_{m',n}$
 for every $m \leq m'$ and $n \in \omega$, and let $X_n =
 \bigcup_{i \leq n} P_{n,i}$. It is easy to check that the $X_n$
 are increasing and closed, and that $\restr{f}{X_n}$ is continuous
 (since $P_{n,i} \cap P_{n,j} = \emptyset$ whenever $i \neq j$). 
 In the rest of this paper, when we will refer to some piecewise
 continuous function we will generally have in mind a function
 with this ``partition'' property. Finally, a third 
equivalent and useful definition is that $X$ can be covered by a (not necessarily increasing) 
countable family  
$P_0, P_1, \dotsc$ of closed sets such that $\restr{f}{P_n}$ is
continuous for every $n \in \omega$.

\begin{defin}
A set $S$ in a metric space is said to be Souslin-$\mathscr{F}$ set
if it belongs to $\mathcal{A}\bPi^0_1$, where $\mathcal{A}$ is the 
usual Souslin operation (see \cite[Definition 25.4]{kechris}).

A metric space $X$ is said to be an \emph{absolute 
Souslin-$\mathscr{F}$ set} if $X$ is a Souslin-$\mathscr{F}$ set 
in the completion of $X$ under its metric.
\end{defin}

Observe that if $X$ is separable then it is an absolute 
Souslin-$\mathscr{F}$ set if and only if it is Souslin, that
is if and only if it is the continuous image of the Baire space
$\Bai$.

Now we are ready to give the statement of the Jayne--Rogers theorem.

\begin{theorem}[Jayne--Rogers]\label{theorJRoriginal}
If $X$ is an absolute Souslin-$\mathscr{F}$ set, then $f\colon X 
\to Y$ is a $\bDelta^0_2$-function if and only if it is piecewise
 continuous.
\end{theorem}

According to the authors of \cite{jaynerogers}, their proof ``even
 in the case when $X$ and $Y$ are separable, is complicated''. 
Sixteen years later, S{\l}awomir Solecki
 provided in \cite{solecki} a new proof of Theorem \ref{theorJRoriginal}
 in the case when $X$ and $Y$ are separable and $X$ is Souslin 
(in fact he proved a much stronger result which refines Theorem 
\ref{theorJRoriginal}), but even in that case the proof was quite
 complicated. Our goal is to provide a simpler proof of Theorem 
\ref{theorJRoriginal}. Our proof is divided into two steps: first we
will prove the nontrivial direction of Theorem \ref{theorJRoriginal}
with the auxiliary assumptions that $X$ is completely metrizable and
$f$ is of Baire class $1$ (see Theorem \ref{theorJR}), and then we
will use a combination of several well-known results to prove Theorem
\ref{theorJRoriginal} as a corollary of Theorem \ref{theorJR}. The
authors would like to thank the anonymous referee for
suggesting a way for removing the condition of separability on the
spaces involved.\\

We will assume $\ZF+\DCR$ throughout the paper (note that the 
Jayne--Rogers' and Solecki's proofs are carried out in $\ZFC$, 
but by a simple absoluteness argument the result must hold also 
in $\ZF+\DCR$). \emph{All spaces considered are metric}.
Our notation will be quite standard: the set of 
the natural numbers will be denoted by $\omega$, while if $X$ is 
any topological space and $A$ is a subset of $X$ we will denote 
the closure of $A$ with $\Cl(A)$. The set of all binary sequences
 of finite length will be denoted by  ${}^{<\omega}2$, and $\Can$ will
 denote the Cantor space.
A function $f \colon X \to Y$
will be said of \emph{Baire class $1$} if it is the pointwise limit
of a sequence of continuous functions $f_n \colon X \to Y$. 
Finally, if $(X,d)$ is any metric
 space, a set $U \subseteq X$ will be called \emph{basic open} if it
 is an open ball of $X$, i.e.\ if $U = \{ x \in X \mid d(x,x_0) < r\}$
 where $x_0 \in X$ and $r \in \RR^+$. For all the other undefined 
symbols and notions we refer the reader to the standard monograph 
\cite{kechris}.

\section{The proof of the Jayne--Rogers theorem}\label{sectionJR}

The main result of this paper is the following theorem, from
which the Jayne--Rogers theorem will follow.

\begin{theorem}\label{theorJR}
Let $X$ and $Y$ be metric spaces such that the metric of  $X$ is complete, 
and let $f\colon X \to Y$ be of Baire
class $1$. If $f$ is a $\bDelta^0_2$-function then it is piecewise
continuous.
\end{theorem}

Recall that if $f\colon X \to Y$ is of Baire class $1$ then it is also
\emph{$\bSigma^0_2$-measurable}, i.e.\ $f^{-1}(U) \in \bSigma^0_2$ for
every open set $U\subseteq Y$, but the converse in general fails. 
Nevertheless, if we require that $X$ is a zero-dimensional absolute
Souslin-$\mathscr{F}$ set then $f$ is 
of Baire class $1$ just in case it is $\bSigma^0_2$-measurable (see 
\cite[Theorem 8]{hansell}). Recall also that a family $\B$ of
subsets of $X$ is said to be \emph{discrete} if $X$ can be covered by
open sets each having a nonvoid intersection with at most one member of
$\B$ (in particular, the elements of $\B$ must be pairwise
disjoint). If $\B$ is a discrete family, then the following facts
easily follow from the definition:

\begin{itemize}
\item $\Cl(\B) = \{ \Cl(B) \mid B \in \B \}$ is discrete;
\item if $\B'$ is a family of subsets of $X$ and there is an injection
  $j \colon \B' \to \B$ such that $B' \subseteq j(B')$ for every $B'
\in \B'$ (e.g.\ if $\B' \subseteq \B$), then $\B'$ is discrete as well;
\item if each $B \in \B$ is closed then $\bigcup \B$ is also closed;
\item if $f \colon X \to Y$ is such that $\restr{f}{B}$ is continuous
  for every $B \in \B$, then $\restr{f}{\bigcup \B}$ is continuous. 
\end{itemize}

The following construction will be used a couple of times: let $g$ be any function defined on a metric space $Z$, let $\F_g$ be the collection of all closed sets $C$ of the completion of $Z$ such that $\restr{g}{(C \cap Z)}$ is continuous, and let $\I_g$\label{I} be the $\sigma$-ideal of the subsets of the fixed completion of $Z$ that can be covered by countably many elements of $\F_g$ (note in particular that $A \subseteq Z$ belongs to $\I_g$ if and only if $\restr{g}{A}$ is piecewise continuous).

\begin{lemma}\label{lemmadiscrete}
  $\I_g$ is closed under discrete unions.
\end{lemma}

\begin{proof}
Let $\B$ a discrete family of subsets of $Z$ and assume $\B \subseteq
\I_g$. Let $F^B_n$ (for $B \in \B$ and $n \in \omega$) be closed sets
such that $B \subseteq \bigcup_n F^B_n$ and $\restr{g}{F^B_n}$ is
continuous. We can assume without loss of generality that $F^B_n
\subseteq \Cl(B)$ (if not simply replace $F^B_n$ by $F^B_n \cap
\Cl(B)$). Put $\F_n = \{ F^B_n \mid B \in \B\}$. Since the function $j
\colon \F_n \to \Cl(\B)$ which maps $F^B_n$ to $\Cl(B)$ is injective, by
the facts aboute discrete families mentioned above we get that
$\Cl(\B)$, and hence also each $\F_n$, must be discrete: but this
implies that $F_n = \bigcup \F_n$ is closed and $\restr{g}{F_n}$ is
continuous. Therefore $\bigcup \B \in \I_g$ because $\bigcup \B
\subseteq \bigcup_{B \in \B} \left( \bigcup_n F^B_n \right) =
\bigcup_n \left( \bigcup_{B \in \B} F^B_n \right) = \bigcup_n F_n$.
\end{proof}

\begin{proof}[Proof of Theorem \ref{theorJRoriginal}]
One direction is trivial. For the other direction, assume toward a
contradiction 
that $f$ is a $\bDelta^0_2$-function but
\emph{not} piecewise continuous. Let $\I = \I_f$ be defined as before (with $Z = X$).
By  \cite[Proposition 3.5]{zeleny}, Lemma \ref{lemmadiscrete} implies that $\I$ is
locally determined, and since it is trivially $\boldsymbol{F}_\sigma$
supported we can apply \cite[Theorem 1.3]{zeleny}: therefore, either
$X \in \I$ or there is $\tilde{X} \subseteq X$ such that $\tilde{X}$ is a  $\bPi^0_2$-subset of the completion of $X$
(hence a completely metrizable space) and $\tilde{X} \notin \I$. Moreover, inspecting
the proof of  \cite[Theorem 1.3]{zeleny} it is easy to check that the $\tilde{X}$
obtained in the second case is also zero-dimensional.
Since the first alternative easily implies that $f$ is piecewise continuous,
we can assume that the second alternative holds and therefore that
$f' = \restr{f}{\tilde{X}}$ is not piecewise continuous. Note that we can assume
also that $f'$ is of Baire class $1$ (otherwise, by \cite[Theorem 8]{hansell} we would have that $f'$ is not 
even $\bSigma^0_2$-measurable and hence not a $\bDelta^0_2$-function), 
and therefore we can apply Theorem \ref{theorJR} to $f'$; this gives
the desired contradiction.
\end{proof}

The strategy for the proof of Theorem \ref{theorJR} will be as follows: we will assume that $f \colon X \to Y$ is of Baire class $1$ (hence also $\bSigma^0_2$-measurable) but not piecewise continuous, and then we will prove that $f$ can not be a $\bDelta^0_2$-function by constructing an open set $\hat{U} \subseteq Y$\label{Ug} such that $f^{-1}(\hat{U})$ is a $\bSigma^0_2$-complete set. To prove that $f^{-1}(\hat{U})$ is $\bSigma^0_2$-complete, we will construct (together with $\hat{U}$) a continuous reduction from the well-known $\bSigma^0_2$-complete set 
\[ S = \{ z \in \Can \mid z(n) \text{ eventually equals }0\} = \{ z \in \Can \mid \exists i \forall j \geq i (z(j)=0)\} \]
 to $f^{-1}(\hat{U})$, i.e.\ a continuous function $g \colon \Can \to X$ such that for all $z \in \Can$
\[ z \in S \iff g(z) \in f^{-1}(\hat{U}).\]
The construction of $\hat{U}$ and $g$ will be carried out by inductively localizing the property of not being piecewise continuous of $f$ to smaller and smaller subsets of $X$.
However, 
before proving Theorem \ref{theorJR} we need a
 couple of technical lemmas.
For the next few results,
 \emph{$X'$ will be an arbitrary subset of  $X$}. 
Given $A,B \subseteq Y$ we will say that $A$ and $B$ are \emph{strongly 
disjoint} if $\Cl(A) \cap \Cl(B) = \emptyset$. Moreover 
if $h \colon X' \to Y$ is any function we put $A^h = h^{-1}(Y \setminus 
\Cl(A))$. Note that for every $A,B \subseteq Y$ one has $(A \cup B)^h
= A^h \cap B^h$. If $h$ is $\bSigma^0_2$-measurable and 
$A,B \subseteq Y$ are strongly disjoint, then we have that if 
$\restr{h}{A^h}$ and $\restr{h}{B^h}$ are both piecewise continuous then 
the whole $h$ is piecewise
continuous. In fact, $A^h$ and $B^h$ is a finite 
$\bSigma^0_2$-covering of $X'$ (by the strongly 
disjointness of $A$ and $B$), which by the reduction property of 
$\bSigma^0_2$ can be refined to a $\bDelta^0_2$-partition $\langle D_0, 
D_1 \rangle$ of $X'$ such that $D_0 \subseteq A^h$, 
$D_1 \subseteq B^h$, and hence both $\restr{h}{D_0}$ and 
$\restr{h}{D_1}$ are piecewise continuous. But if $h'\colon X' \to Y$ is 
such that for some $\bDelta^0_2$-partition $\seq{D'_n}{n \in \omega}$ of 
$X'$ we have that $\restr{h'}{D'_n}$ is piecewise continuous  for every 
$n$, then $h'$ is piecewise continuous on the whole $X'$: therefore $h$ 
is piecewise continuous as well.

Now let $h \colon X' \to Y$ be a $\bSigma^0_2$-measurable function, $x \in 
X'$, and $A$ be any subset of $Y$. We say that $x$ is \emph{$h$-irreducible outside $A$} 
if for every open neighborhood $V \subseteq X'$ of $x$ the function 
$\restr{h}{A^h \cap V}$ is not piecewise continuous, otherwise we 
say that $x$ is \emph{$h$-reducible outside $A$}. In our proofs the set $A$ will be often of the form $A = U_0 \cup \dotsc \cup U_n$ with $U_0, \dotsc, U_n$ a sequence of pairwise strongly disjoint open sets. Notice that if $x$ is $h$-irreducible outside $A$ then $x \in \Cl(A^h)$, as otherwise $A^h \cap V = \emptyset$ for some open neighborhood $V$ of $x$ and therefore $\restr{h}{A^h \cap V}$ would be trivially (piecewise) continuous. Moreover, if there are $x$ and $A$ such that $x$ is $h$-irreducible outside $A$ then clearly $h$ can not be piecewise continuous. Finally, it is easy to check that if $x$ is $h$-irreducible outside $A$ and $A' \subseteq A$ 
then $x$ is also $h$-irreducible outside $A'$, and that if $X'' 
\subseteq X'$ and  $x \in X''$ is $h'$-irreducible outside $A$ (where $h'=\restr{h}{X''}$) 
then $x$ is also $h$-irreducible outside $A$.

\begin{lemma}\label{lemmacrucial}
Suppose $h\colon X' \to Y$ is a $\bSigma^0_2$-measurable function and 
$U_0, \dotsc, U_n \subseteq Y$ are basic open sets of $Y$ such that 
${\rm range}(h) \cap \Cl(U_i) = \emptyset$ for every $i \leq n$.
Then $h$ is \emph{not} piecewise continuous if and only if {\rm ($*$)} 
there is an $x \in X'$ and a basic open set $U \subseteq Y$ strongly 
disjoint from $U_0, \dotsc, U_n$ such that $h(x) \in U$ and $x$ is $h$-irreducible outside
$U$.
\end{lemma}

\begin{proof}
Put $C = \Cl(U_0) \cup \dotsc \cup \Cl(U_n)$.
We will prove that $h$ is piecewise continuous if and only if ($*$)
does not hold. If $h$ is piecewise continuous then the same must hold
for $\restr{h}{X''}$  
where $X''$ is any subset of $X'$, therefore one direction is trivial. For 
the other direction, assume toward a contradiction that ($*$) does not 
hold, i.e.\ for every $x \in X'$ and every open set $U \subseteq Y$ 
strongly disjoint from $C$ such that $h(x) \in U$ we have that $x$
is $h$-reducible outside $U$, that is there is some open 
neighborhood $V \subseteq X'$ of $x$ such that $\restr{h}{U^h \cap V}$ 
is piecewise continuous. Since $X$ is a metric space, and hence also paracompact, 
let $\B = \bigcup_n \B_n$ be a base for the topology of $X$ such that each $\B_n$ is
discrete (see \cite{engelkind}). Then let $Q_n$ be the union of the elements of $\B_n$ which belongs to $\I = \I_h$, so that each $Q_n$ belongs to $\I$ by Lemma \ref{lemmadiscrete}. Finally put $Q = \bigcup_n Q_n$ and notice that $\restr{h}{Q}$ is piecewise continuous since $Q \in \I$, and that $Q$ is open and contains as a subset each open set $W$ for which $\restr{h}{W}$ is piecewise continuous.
We claim that $\restr{h}{X' \setminus Q}$ is continuous (from this
easily follows that $h$ is piecewise continuous). Suppose otherwise, so that given 
any $x \in X' \setminus Q$ and any open set $U \subseteq Y$ such 
that $h(x) \in U$ there is no open neighborhood
$V$ of $x$ such that $h(V \cap (X' \setminus Q)) \subseteq U$.
Fix such an $x$ and $U$, and let $U'\subseteq Y$ be basic open, strongly disjoint from $C$, and such 
that $h(x) \in U'$ and $\Cl(U') \subseteq U$ ($U'$ exists since $Y$ is metric). Let $V 
\subseteq X'$ be given by the failure of ($*$) on the inputs $x$ and
$U'$: by our hypothesis there is
$x' 
\in V \cap (X' \setminus Q)$ such that $h(x') \notin \Cl(U')$, and clearly we can find a basic open
 $U'' \subseteq Y$  strongly 
disjoint from $U'$ and $C$, and such that $h(x') \in U''$. Let $V' 
\subseteq X'$ be the open set given by the failure of ($*$) on inputs $x'$ 
and $U''$. Since $V$ and $V'$ have been chosen in such a way that 
$\restr{h}{(U')^h \cap V}$ and $\restr{h}{(U'')^h \cap V'}$ are piecewise continuous,
and since $\{(U')^h \cap V,(U'')^h \cap V'\}$ is a $\bSigma^0_2$-covering
of $V \cap V'$, by the strong disjointness of $U'$ and $U''$ we must have that 
$\restr{h}{V \cap V'}$ is piecewise continuous,
 and therefore $V \cap V' \subseteq Q$: 
but this implies that $x' \in Q$, a contradiction! 
\end{proof}

\begin{lemma}\label{lemmagood}
Let $h \colon X' \to Y$ be a $\bSigma^0_2$-measurable function, $x  \in X'$, 
$A \subseteq Y$, and $U_0, \dotsc, U_n$ be a sequence of pairwise strongly disjoint 
open subsets of $Y$. If $x$ is $h$-irreducible outside $A$ 
then there is at most one 
$i\leq n$ such that $x$ is \emph{$h$-reducible} outside $A \cup U_i$.
\end{lemma}

\begin{proof}
Assume that $i \leq n$ is such that $x$ is $h$-reducible outside  $A \cup U_i$, 
i.e.\ that there is an open neighborhood $V \subseteq X'$ of $x$ such that 
$\restr{h}{(A \cup U_i)^h \cap V}$ is piecewise 
continuous. If there were some $j \neq i$ with the same property, then 
there must be some open neighborhood $W \subseteq X'$ of $x$ such that 
$\restr{h}{(A \cup U_j)^h \cap W}$ is piecewise 
continuous. But since $U_i$ and $U_j$ are strongly disjoint, this would 
imply that $\restr{h}{A^h \cap V \cap W}$ is 
piecewise continuous as well, and thus $V \cap W$ would contradict the 
fact that $x$ is $h$-irreducible outside $A$.
\end{proof}

Finally observe that if $f \colon X \to Y$ is the pointwise limit of a 
sequence  of functions $\seq{f_m \colon X \to Y}{m \in \omega}$, 
then we have the following property: if $x \in X$ and $U_0, U_1, \dotsc$ 
are pairwise disjoint open sets such that for infinitely many  $n$'s there 
is an $m$ for which $f_m(x) \in U_n$, then $f(x) \notin U_n$ for each $n$ 
(otherwise, $f_m(x) \in U_n$ for all but finitely many $m$'s contradicting 
our hypothesis).

Now we are ready  to prove Theorem \ref{theorJR}. The proof essentially uses
recursively Lemma \ref{lemmacrucial} applied to smaller and smaller subspaces
of $X$ to construct some sequences, and Lemma \ref{lemmagood} will guarantee that
at each stage the 
construction can be carried out. This is the reason for which we have proved
both the lemmas for arbitrary functions $h$ 
with domain an arbitrary subset $X'$ of $X$: in fact we will generally apply them to the restriction of the original
function $f$ to some subset of $X$, that is with $h = \restr{f}{X'}$.

\begin{proof}[Proof of Theorem \ref{theorJR}]
Assume that $f\colon X \to Y$ is of Baire class $1$ (hence also 
$\bSigma^0_2$-measurable) but not piecewise continuous, and let 
$\seq{f_n}{n \in \omega}$ be a sequence of continuous functions which 
pointwise converges to $f$. As explained on page \pageref{Ug}, we will inductively construct an open set $\hat{U} 
\subseteq Y$ and a continuous reduction $g \colon \Can \to X$ from $S = \{ z \in \Can \mid \exists i \forall j \geq i (z(j)=0)\}$
to $f^{-1}(\hat{U})$.
The function $g$ will be defined using a \emph{weak}
Cantor scheme $\seq{V_s}{s \in {}^{<\omega} 2}$
(that is a classical Cantor scheme in which we drop
the condition $V_{s \conc 0} \cap V_{s \conc 1} = \emptyset$)
 such that for every $s,t \in {}^{<\omega} 2$ we have:
\begin{enumerate}[1)]
\item $V_s$ is an open subset of $X$;
\item if $s \subsetneq 
t$ then $\Cl(V_t) \subseteq V_s$;
\item $\diam(V_s) \leq 2^{-\leng(s)}$.
\end{enumerate}
It is straightforward to check that, given such a scheme, the function $g  
\colon \Can \to X \colon z \mapsto \bigcap_{n \in \omega} 
V_{\restr{z}{n}}$ is well-defined (by the completeness of $X$) 
and continuous (in fact it is 
Lipschitz with constant $1$). 

The construction will be carried out by recursion on the rank of $s 
\in {}^{<\omega}2$ with respect to the order $\preceq$ defined by
\[ s \preceq t \iff \leng(s) < \leng(t) \vee (\leng(s) = \leng(t) \wedge 
s \leq_{\rm lex} t),\]
where $\leq_{\rm lex}$ is the usual lexicographical order on ${}^{<\omega}2$ (the strict part
of $\preceq$ will be denoted by $\prec$). 
In fact we will define, together with a scheme $\seq{V_s}{s \in {}^{<\omega}2}$ 
with the properties above, a sequence $\seq{x_s}{s \in {}^{<\omega}2}$ of points
of $X$ and a sequence $\seq{U_s}{s \in {}^{<\omega}2}$ of subsets of $Y$
such that for 
every $s \in {}^{<\omega}2$:
\begin{enumerate}[i)]
\item $x_s \in V_s$;
\item $f(x_s) \in U_s$;
\item $U_s$ is basic open and for every $t \in {}^{<\omega}2$ we have that 
$U_s$ and $U_t$ are either equal or strongly disjoint;
\item there is some $m \in \omega$ such that $f_m (V_s) \subseteq 
U_s$;
\item $x_t$ is $f$-irreducible outside $A$ for every $t \preceq s$, where 
$A=\bigcup_{u \preceq s} U_{u}$;
\item if the last digit of $s$ is $1$ then $U_s \neq U_t$ for every 
$t \prec s$ 
(and therefore, in particular, for every $t \subsetneq s$).
\end{enumerate}

As already noted, to construct these sequences we will recursively apply
Lemma \ref{lemmacrucial} to the restriction of $f$ to smaller and smaller
pieces. 

At the first stage, let $x$ and $U$ be given as in Lemma \ref{lemmacrucial}
applied to the whole $f$, 
and let $V = f_m^{-1}(U)$ where $m \in \omega$ is such that $f_m(x) \in 
U$ (such an $m$ must exists by the fact that $f$ is the limit of the 
$f_n$'s). Then put $V_\emptyset=V$, $x_\emptyset = x$ and 
$U_\emptyset = U$. Now let $s \neq \emptyset$ and suppose we have 
defined $V_t$, $x_t$ and $U_t$ for $t \prec s$. Put $s^- = \restr{s}{(\leng(s)-1)}$. If the last digit 
of $s$ is a $0$, then simply put $V_s = W$, $x_s = x_{s^-}$ and 
$U_s = U_{s^-}$, where $W$ is any open set such that $\Cl(W) 
\subseteq V_{s^-}$, $x_s \in W$ and $\diam(W) \leq 2^{-\leng(s)}$. 
Otherwise the last digit of $s$ is $1$: by the inductive hypothesis, condition v) 
implies that $h_0 = \restr{f}{A^f \cap V_{s^-}}$, where $A = \bigcup_{t
  \prec s}  U_t$,
is not piecewise continuous (otherwise, since $x_{s^- \conc 0} \in
V_{s^- \conc 0} 
\subseteq V_{s^-}$, $x_{s^- \conc 0}$
should be $f$-reducible outside $A$). 

\begin{claim}
There are $x_s \in V_{s^-}$ and $U_s \subseteq Y$ such that $f(x_s) \in
U_s$, $U_s$  
is basic open and strongly disjoint from $A$ (which in particular implies $U_s \neq
U_t$ for every 
$t \prec s$), and $x_t$ is $f$-irreducible outside $A\cup U_s$
for every  $t \preceq s$.
\end{claim}

\begin{proof}[Proof of the Claim]
Let $k = |\{ t \in {}^{<\omega}2 \mid t \prec s \}  |$. Using Lemma \ref{lemmacrucial}, for $j \leq k+1$ recursively construct $x_j$ and $U_j$  such that each $x_j$ belongs to $V_{s^-}$, $f(x_j) \in U_j$, $U_j$ is strongly disjoint from $A \cup U_{<j}$ (where $U_{<j} = \emptyset$ if $j=0$ and $U_{<j} = \bigcup_{i<j}U_i$ otherwise), and $x_j$ is $h_j$-irreducible outside $A \cup U_{<j} \cup U_j$ (hence in particular $x_j$ is  $f$-irreducible outside $A \cup U_j$), where $h_0$ is as before and $h_{j+1} =\restr{ h_j}{(A \cup U_{<(j+1)})^f}$. Now notice that there must be $\bar{\jmath} \leq k+1$ such that the claim is satisfied with $x_s = x_{\bar{\jmath}}$ and $U_s = U_{\bar{\jmath}}$: if not, by the pigeonhole principle there should be $j \neq j' \leq k+1$ and $t \prec s$ such that $x_t$ is $f$-reducible both outside $A \cup U_j$ and $A \cup U_{j'}$, contradicting Lemma \ref{lemmagood}.
\renewcommand{\qedsymbol}{$\square$ \textit{Claim}}
\end{proof}

Let $W \subseteq X$ be an open neighborhood of $x_s$ such that ${\rm 
diam}(W) \leq 2^{-\leng(s)}$, $\Cl(W) \subseteq V_{s^-}$ and $f_m (W) 
\subseteq U_s$ for some $m$,  and define $V_s = W$. This completes the recursive
definition of the sequences required.

It is easy to check that the scheme $\seq{V_s}{s \in {}^{<\omega}2}$ and
the sequences $\seq{x_s}{s \in {}^{<\omega}2}$ and $\seq{U_s}{s \in {}^{<\omega}2}$
constructed in this way are as required, i.e.\ that they satisfy 1)--3) 
and i)--vi). Now put $\hat{U} = \bigcup_{s \in {}^{< \omega} 2} U_s$,
and let $g\colon \Can \to X$ be obtained from $\seq{V_s}{s \in {}^{<\omega}2}$ 
as described 
above.  We have only to check that $g$ is a reduction of $S$ to 
$f^{-1}(\hat{U})$. Let $\seq{U_k}{k \in \omega}$
be an enumeration without repetitions of $\seq{U_s}{s \in {}^{<\omega}2}$, 
so that by condition iii) the $U_k$'s are pairwise 
disjoint and $\hat{U} = \bigcup_{k \in \omega} 
U_k$. If $z \in S$, then for some $\bar{n} \in \omega$ 
we will have that $x_{\restr{z}{m}} = x_{\restr{z}{\bar{n}}} = \bar{x}$ 
for every $m \geq \bar{n}$, therefore $g(z) = \bar{x}$ and $f(g(z)) = 
f(\bar{x}) \in U_{\restr{z}{\bar{n}}} \subseteq \hat{U}$. Assume now 
$z \notin S$: by conditions vi) and iv), for infinitely many $k$'s there
is some $m \in \omega$ such that $f_m(g(z)) \in U_k$ (since $g(z) \in 
V_{\restr{z}{n}}$ for every $n \in \omega$), and therefore $f(g(z))
\notin \hat{U}$ by the observation preceding this proof.
\end{proof}

\end{document}